\documentclass[11pt]{amsart}
\usepackage{lmodern,graphicx}
\usepackage{latexsym}
\usepackage{amsmath, amssymb}
\usepackage{amsfonts}
\usepackage{pgfplots}
\usepackage{tikz}
\usepackage[T1]{fontenc}
\usepackage[german,english]{babel}

\newcommand{\Z}{\ensuremath{\mathbb{Z}}}

\newtheorem{theor}{Theorem}

\newtheorem{thm}{Theorem}[section]

\newtheorem{lemma}[thm]{Lemma}
\newtheorem{cor}[thm]{Corollary}
\newtheorem{corollary}{Corollary}

\newtheorem{proposition}[thm]{Proposition}
\newtheorem*{remark}{Remark}

\begin{document}

\title{The class $\mathcal{MN}$ of groups in which all maximal subgroups are normal}
\author{Aglaia Myropolska}
\thanks{The author acknowledges the support of the Swiss National Science Foundation, grant 200021\_144323 and P2GEP2\_162064.}
\address{Laboratoire de Math\'ematiques, 
Universit\'e Paris-Sud 11, Orsay, France}
\email{aglaia.myropolska@math.u-psud.fr}
\vspace{1cm}
\date{\today}
\begin{abstract}
We investigate the class $\mathcal{MN}$ of groups with the property that all maximal subgroups are normal. The class $\mathcal{MN}$ appeared in the framework of the study of potential counter-examples to the Andrews-Curtis conjecture. 
In this note we give various structural properties of groups in $\mathcal{MN}$ 
and present examples of groups in $\mathcal{MN}$ and not in $\mathcal{MN}$.

\end{abstract}

\maketitle

\section{Introduction}
The class $\mathcal{MN}$ was introduced in \cite{Myropolska} as the class of groups with the property that all maximal subgroups are normal. The study of $\mathcal{MN}$ was motivated by the analysis of potential counter-examples to the Andrews-Curtis conjecture \cite{MR0173241}. It was shown in \cite{Myropolska} that a finitely generated group $G$ in the class $\mathcal{MN}$ satisfies the so-called ``generalised Andrews-Curtis conjecture'' (see \cite{MR2195451} for the precise definition) and thus cannot confirm potential counter-examples to the original conjecture. 

Apart from its relation to the Andrews-Curtis conjecutre, the study of the class $\mathcal{MN}$ can be interesting on its own. Observe that if a group $G$ belongs to $\mathcal{MN}$ then all maximal subgroups of $G$ are of finite index. The latter group property has been considered in the literature for different classes of groups. 
For instance in the linear setting, Margulis and Soifer \cite{MR613853} showed that all maximal subgroups of a finitely generated linear group $G$ are of finite index if and only if $G$ is virtually solvable.
The above property also was considered for branch groups, however the results in this direction are partial and far from being as general as for linear groups. 
For instance, the question whether all maximal subgroups of the first Grigorchuk group $\Gamma$ have finite index was open for a long time before it was answered positively by Pervova in \cite{MR1841763}. There were several generalizations of this result to some other branch $p$-groups, for instance the Gupta-Sidki $p$-groups (see \cite{MR2197824,Maxsub_spinal}). All these approaches are inherently related to the structure of the specific branch group. One could expect that finitely generated branch groups only have maximal subgroups of finite index, but Bondarenko \cite{MR2727305} exhibited a finitely generated branch group (not a $p$-group) that has maximal subgroups of infinite index.

Maximal subgroups of a given group $G$ are also related to its primitive actions: any primitive action of $G$ in $\mathcal{MN}$ coincides with its action on the finite cyclic quotient $G/M$ for some maximal subgroup $M$. 


A straightforward example of a group belonging to $\mathcal{MN}$ is an abelian, and more generally, a nilpotent group. Other examples also include the above mentioned finitely generated branch $p$-groups, namely the family of Grigorchuk groups and the Gupta-Sidki $p$-groups, as well as others.  A somehow exceptional example is provided by the construction of Ol'shanski in \cite{MR1191619} of a countable (not finitely generated) group $\mathcal{G}$ without maximal subgroups. The construction can be adapted so that $\mathcal{G}$, which clearly belongs to $\mathcal{MN}$, contains the free group $F_2$ as a subgroup. 
However, a free group and more generally any free product of finitely generated groups do not belong to $\mathcal{MN}$ (see Section \ref{freeproduct}). 

We would like to point out that in the finitely generated group setting, when the existence of a maximal subgroup is guaranteed by Zorn's lemma, we have no example of a (finitely generated) non-amenable group belonging to $\mathcal{MN}$. Moreover, we still do not know whether the class $\mathcal{MN}$ is closed under taking subgroups when the ambient group is finitely generated.

Below we summarize some properties of groups in the class $\mathcal{MN}$.

\begin{theor}
\label{properties}
Let $G$ be a group. Then the following statements are equivalent.
\begin{enumerate}
\item $G$ is in $\mathcal{MN}$.
\item For every normal subgroup $N\lhd G$, the quotient $G/N$ is in $\mathcal{MN}$.
\item All maximal subgroups of $G$ are of finite index and all finite quotients of $G$ are nilpotent.
\item $[G,G]\leq \Phi(G)$, where $[G,G]$ is the commutator subgroup of $G$ and $\Phi(G)$ is the Frattini subgroup of $G$.
\end{enumerate}
\end{theor}

It follows, in particular, that if $G$ and $H$ are finitely generated groups in $\mathcal{MN}$ then their direct product $G\times H$ is in $\mathcal{MN}$ as well (see Corollary \ref{direct}). 
We further list a few corollaries for \emph{finitely generated groups} in $\mathcal{MN}$ (see Section \ref{further-prop} for more properties). 
\begin{corollary}
Let $G$ be a finitely generated group in $\mathcal{MN}$. Then $G$ does not contain the free group $F_n$ with $n\geq 2$ as a subgroup of finite index.
\label{free-index}
\end{corollary}

\begin{corollary}
Let $G$ be a finitely generated linear group in $\mathcal{MN}$. Then $G$ is nilpotent.
\label{linear}
\end{corollary}

This paper is organised as follows. Section $2$ contains proofs of Theorem~\ref{properties}, Corollary~\ref{free-index} and Corollary~\ref{linear}, as well as other properties and characterizations of groups in $\mathcal{MN}$. Section~$3$ contains examples of groups in $\mathcal{MN}$ and not in $\mathcal{MN}$. And in Appendix~\ref{appendix} we discuss the relation of the class $\mathcal{MN}$ to the Andrews-Curtis conjecture. 

\medskip
\noindent \textbf{Acknowledgements:} It is a pleasure to thank Pierre de la Harpe and Tatiana Nagnibeda for the encouragement to carry on this work, for numerous discussions and useful references, and Denis Osin for the explanation on a construction by Ol'shanski.

\section{Properties of groups in $\mathcal{MN}$}

\subsection{Criterion in terms of generating and normally generating sets}

\label{section-obs} We will first present a criterion for a finitely generated group to be in $\mathcal{MN}
$ in terms of generating and normally generating sets. 

\smallskip
We recall that the \emph{Frattini subgroup} $\Phi(G)$ of a group $G$ is the intersection of all maximal subgroups of $G$, and $\Phi(G)=G$ if $G$ does not have maximal subgroups. Note that if a non-trivial group $G$ is finitely generated then Zorn's lemma implies that there exists a proper maximal subgroup of $G$.

Recall the following well-known properties of the Frattini subgroup.
\begin{lemma}
\label{frattini}
Let $G$ be a group. 
\begin{enumerate}
\item \cite[5.2.12]{MR648604}  Frattini subgroup $\Phi(G)$ is equal to the set of non-generators of $G$, i.e. if $g\in \Phi(G)$ and $\langle g,X\rangle=G$ then $\langle X\rangle=G$. 
\item \cite{MR1218122} Let $G=\langle x_1,...,x_n\rangle$ and $\varphi_1,...,\varphi_n \in \Phi(G)$. Then \begin{center}$\langle x_1\varphi_1,\dots,x_n\varphi_n\rangle=G$.\end{center}
\end{enumerate}
\end{lemma}

\begin{proposition} Let $G$ be a finitely generated group. Then $G$ is in $\mathcal{MN}
$ if and only if all normally generating sets of $G$ are generating sets of $G$. 
\label{obs}
\end{proposition}

Recall, that a subgroup $H\leq G$ is \emph{normally generated} by a subset $S\subseteq G$ if $H$ is the smallest normal subgroup of $G$ containing $S$. We denote the subgroup normally generated by $S$ by $\ll S\gg^G$ or simply $\ll S\gg$.

\begin{proof}[Proof of Proposition \ref{obs}]

$\Rightarrow$. Let $G$ be in $\mathcal{MN}
$ and assume by contradiction that there is a normally generating set $S$ which is not a generating set. Since $G$ is finitely generated, every proper subgroup is contained in some proper maximal subgroup \cite{N37}, therefore $\langle S\rangle\leq M<G$ for some proper maximal subgroup $M$ . Then $\ll S\gg=G\leq ~\ll M \gg=M< G$. We obtain a contradiction.

$\Leftarrow$. Conversely, we will prove that if all normally generating sets are generating sets then all maximal subgroups are normal. We prove the equivalent statement: if there is a maximal subgroup $M$ which is not normal, then there exists a normally generating set which is not a generating set. We take as a normally generating set $S=M$. Then $\ll S\gg=G$ since $M$ is maximal and $\langle S\rangle\neq G$ since $M$ is proper. 
\end{proof}

\subsection{Free group and free products do not belong to $\mathcal{MN}$} \label{freeproduct} It follows from Section~\ref{section-obs} that the free group $F_2=\langle x, y\rangle$ of rank $2$ is not in $\mathcal{MN}
$. For this notice that $\ll y^{-1}xy, x^{-1}yx \gg=F_2$, however $\langle y^{-1}xy, x^{-1}yx \rangle \neq F_2$.

More generally, a free product of a finite number of non-trivial finitely generated groups $A_1, \dots, A_k$ does not belong to the class $\mathcal{MN}$. To see this, first suppose that $k=2$ and suppose that $A_1$ is generated by $\{a_1, \dots, a_k\}$ and $A_2$ is generated by $\{b_1, \dots, b_m\}$. The set $$\{b_1^{-1} a_1 b_1, b_1^{-1} a_2 b_1,\dots, b_1^{-1}a_k b_1^{-1}, a_1^{-1}b_1a_1,a_1^{-1}b_2a_1, \dots, a_1^{-1}b_m a_1\}$$ normally generates the group $A_1\ast A_2$ but does not generate it; it follows that $A_1\ast A_2$ does not belong to $\mathcal{MN}$. More generally, a free product $A_1\ast \dots \ast A_k$ can be seen as a free product of $A_1$ and $A_2\ast\dots\ast A_k$ and we use the argument for $k=2$.

Notice that the infinite dihedral group $D_{\infty}$, being isomorphic to $\Z/2\Z~\ast~\Z/2\Z$, is not in $\mathcal{MN}$. 

\subsection{Proof of Theorem \ref{properties}}

\begin{proof}[Proof of Theorem \ref{properties}]
The implication $2\Rightarrow 1$ is obvious. 

$1\Rightarrow 2$.  
Consider the natural projection $\pi: G\rightarrow G/N$. If $M$ is a maximal subgroup of the quotient $G/N$ then its preimage $\pi^{-1}(M)$ is a maximal subgroup of $G$. Moreover, since $G$ is in $\mathcal{MN}
$ then $\pi^{-1}(M)$ is normal in $G$ and it follows that $M$ is normal in $G/N$. 

$1\Rightarrow 3$.  Suppose that $G$ belongs to $\mathcal{MN}$. Then in particular every maximal subgroup $M$ of $G$ is of finite index. Indeed, by the lattice theorem (also called the fourth isomorphism theorem) there is a one-to-one correspondence between subgroups of $G$ containing $M$ and those of $G/M$. Since there are no proper subgroups in $G$ containing $M$ we deduce that $G/M$ does not have proper subgroups which implies that $G/M$ is a cyclic group of prime order.

Moreover, if $Q$ is a finite quotient of $G$, then $Q$ belongs to $\mathcal{MN}$ by Theorem \ref{properties}. A finite group is in $\mathcal{MN}$ if and only if it is nilpotent \cite[5.2.4]{MR648604}.

$3\Rightarrow 1$. Let $M$ be a maximal subgroup of $G$. By the assumption, $M$ is of finite index of $G$. Let $K$ be the normal core of $M$, that is $K=\cap_{g\in G} g^{-1}Mg$. The subgroup $K$ is normal and of finite index in $G$. The quotient $G/K$ is finite and therefore nilpotent. It follows that the maximal subgroup $M/K$ in $G/K$ is normal. We conclude that $M$ is normal in $G$. 

$1\Rightarrow 4$. If $G$ is in $\mathcal{MN}$ then for any maximal subgroup $M$ the quotient $G/M$ is isomorphic to $\mathbb{Z}/p\mathbb{Z}$ by the lattice theorem (also called the fourth isomorphism theorem), in particular, $G/M$ is abelian. Therefore $[G,G]\leq M$ and we conclude that $[G,G]\leq \Phi(G)$.

$4\Rightarrow 1$. Let $M$ be a maximal subgroup of $G$. Since the derived subgroup $[G,G]$ is contained in $\Phi(G)$, it is also contained in $M$. Consider $M'=\cap_{g\in G}g^{-1}Mg$, a normal subgroup of $G$. Since $[G,G]$ is normal in $G$, it follows that $[G,G]\leq M'$. Notice that $M/M'$ is a subgroup in the abelian group $G/M'$ and hence normal. It follows that $M$ is normal in $G$. 
\end{proof}

\begin{remark} \emph{The condition on nilpotency of finite quotients in Theorem \ref{properties}.$3$  is essential: indeed, all maximal subgroups of the Lamplighter group $\mathcal{L}=\mathbb{Z}/2\mathbb{Z}\wr \mathbb{Z}$ are of finite index, however there are maximal subgroups of $\mathcal{L}$ which are not normal \cite{MR3278388}. }
\end{remark}


\subsection{Criterion for $p$-groups}
\begin{proposition}
\label{sbg} Let $G$ be a $p$-group and let $H$ be a subgroup of finite index in $G$. If $H\in \mathcal{MN}$ then $G\in \mathcal{MN}$.
\end{proposition}
\begin{proof}
Since $H\in \mathcal{MN}$ then all maximal subgroups of $H$ are of finite index. It follows that all maximal subgroups of $G$ are of finite index (see, for instance, \cite{MR2009443}[Lemma 1]). 

Further, let $M$ be a maximal subgroup of $G$. By above, $M$ is of finite index in $G$. Consider the normal subgroup of finite index $K=\cap_{g\in G} g^{-1}Mg\lhd G$. Clearly, $M/K$ is a maximal subgroup of the quotient $G/K$. Use that $G/K$ is nilpotent to conclude that $M/K$ is normal in $G/K$. The proposition follows.
\end{proof}

\begin{remark}
\emph{
Proposition \ref{sbg} does not hold in whole generality. Indeed, the dihedral group $D_{\infty} \notin \mathcal{MN}$ contains $\mathbb{Z} \in \mathcal{MN}$ of finite index.}
\end{remark}

\subsection{Proofs of corollaries}. 
\begin{proof}[Proof of Corollary \ref{free-index}]
We will prove the corollary by showing that the equivalent statement holds: if a finitely generated group $G$ contains a free group $F_n$ with $n\geq 2$ as a subgroup of finite index then $G$ does not belong to $\mathcal{MN}$.

Let $G$ be a finitely generated group containing a subgroup $H$ of finite index isomorphic to the free group $F_n$ with $n\geq 2$. Denote by $K$ the normal subgroup of finite index in $G$ defined by $K=\cap_{g\in G} g^{-1}Hg$. Being a subgroup of $H$, $K$ is a free group of finite rank. Since $G$ is non-amenable it follows that $K$ is isomorphic to a free group $F_m$ with $m\geq 2$.

Let $K^r$ with $r\geq 2$ be the characteristic subgroup of $K$ generated by all $r$-th powers of elements of $K$. Since $K$ is normal in $G$ it follows that $K^r$ is also normal in $G$. Notice that the quotient $G/K^r$ contains as a subgroup $K/K^r$, and the latter is isomorphic to the Burnside group $B(m, r)$ of rank $m$ and exponent $r$. 

Let $r=6$. The group $B(m, 6)$ is finite by  \cite{MR0102554}. In particular, it follows that $K^6$ is of finite index in $K$ and thus of finite index in $G$. 
Further, observe that $B(m,6)$ is not nilpotent. First, for $m=2$, the group $B(2,6)$ contains the subgroup isomorphic to the dihedral group $D_{12}$, and the latter is not nilpotent (see \cite{MR1776769} for the details on subgroups and presentation of $B(2,6)$). For $m\geq 3$, notice that $B(m,6)$ admits the natural epimorphism on the non-nilpotent group $B(2,6)$.

We conclude that the quotient $G/K^6$ is finite and moreover it contains non-nilpotent subgroup $K/K^6$. Therefore $G$ admits a finite non-nilpotent quotient $G/K^6$ and we deduce that $G$ does not belong to $\mathcal{MN}$ by Theorem \ref{properties}.
\end{proof}

\begin{proof}[Proof of Corollary \ref{linear}]
Since $G$ is in $\mathcal{MN}
$ then all maximal subgroups of $G$ are of finite index, as observed in Theorem \ref{properties}. The result of \cite{MR613853} tells us that all maximal subgroups of a finitely generated linear group $G$ are of finite index if and only if $G$ is virtually solvable. Furthermore, since all finite quotients of $G$ are nilpotent, we deduce that $G$ is solvable. Finally, to conclude that $G$ is nilpotent we use the result from \cite{MR0251120}, saying that a finitely generated solvable group such that all of its finite homomorphic images are nilpotent is nilpotent.
\end{proof}
In particular, it implies that a finitely generated linear group is in $\mathcal{MN}$ if and only if it is nilpotent.

\subsection{Further properties of \emph{finitely generated} groups in $\mathcal{MN}$} 
\label{further-prop}In this section we continue the discussion of the properties of finitely generated groups in the class $\mathcal{MN}$. First, we observe that $\mathcal{MN}$ is closed under taking direct products. Then we show that finitely generated groups with infinitely many ends, and thus certain  amalgamated products of groups and HNN extensions, do not belong to $\mathcal{MN}$. 

\begin{cor}
Let $G_1, G_2$ be finitely generated groups in $\mathcal{MN}
$. Then $G=G_1\times G_2$ is in $\mathcal{MN}
$.
\label{direct}
\end{cor}
\begin{proof}
The Frattini subgroup of a direct product of finitely generated groups is the direct product of Frattini subgroups of the direct factors \cite{MR0122885}, \emph{i.e.} $\Phi(G)=\Phi(G_1)\times \Phi(G_2)$. 

Consider the derived subgroup $[G,G]=[G_1,G_1]\times [G_2, G_2]$. By assumption $G_1$ and $G_2$ are in $\mathcal{MN}
$ therefore $[G,G]\leq \Phi(G)$. We conclude that $G$ is in $\mathcal{MN}
$ by Theorem \ref{properties}.
\end{proof}

\begin{proposition}
Let $G$ be a finitely generated group with infinitely many ends. Then $G$ does not belong to $\mathcal{MN}$.
\label{ends}
\end{proposition}
\begin{proof}
Suppose by contradiction that $G$ belongs to $\mathcal{MN}$.
Then $[G,G]\leq \Phi(G)$ by Theorem~\ref{properties} and hence the quotient $G/\Phi(G)$ is abelian. Further, since $G$ is a finitely generated group with infinitely many ends, it follows that $G$ has a finite nilpotent Frattini subgroup \cite[Corollary $8.6$]{MR2377495}. Hence $G$ is abelian-by-nilpotent and thus solvable. We use the result of Wehfritz \cite{MR0251120}, saying that a finitely generated solvable group such that all of its finite quotients are nilpotent is nilpotent, to conclude that $G$ is nilpotent.

By assumption $G$ has infinitely many ends. It follows from the Stallings' theorem (see \cite[$5.A.10$]{MR0415622}) that \emph{either} $G$ admits a splitting as $G=H\ast_C K$ as a free product with amalgamation  over a proper finite subgroup $C$ and $\max\{|H:C|,|K:C|\}> 2$, \emph{or} $G$ admits a splitting $G=\langle H,t | t^{-1}C_1t=C_2\rangle$ where $C_1$ and $C_2$ are isomorphic finite proper subgroups of $H$. Observe that in both cases, $G$ contains a free group $F_2$ and thus cannot be nilpotent. We obtain a contradiction with $G$ being in $\mathcal{MN}$.
\end{proof}

\begin{cor} Let $G$ be a finitely generated group such that one of the following holds:
\begin{itemize}
\item $G$ admits a splitting $G=H\ast_C K$ as a free product with amalgamation where $C$ is a finite group with $C\neq H$ and $C\neq K$;
\item $G$ admits a splitting $G=\langle H,t | t^{-1}C_1t=C_2\rangle$ where $C_1$ and $C_2$ are isomorphic finite proper subgroups of $H$.
\end{itemize}
Then $G$ does not belong to $\mathcal{MN}$.
\end{cor}
\begin{proof}
We will first treat the special case when $G$ admits a splitting $G=H\ast_C K$ as a free product with amalgamation and $|H:C|=|K:C|=2$. Notice that $G$ surjects onto the infinite dihedral group $\mathbb{Z}/2\mathbb{Z} \ast \mathbb{Z}/2\mathbb{Z}\cong D_{\infty}$. We conclude that $G$ does not belong to $\mathcal{MN}$ by Theorem \ref{properties} and Section \ref{freeproduct}.

If $G$ satisfies the assumption of the corollary, then either it has infinitely many ends by the Stallings' theorem (see \cite[$5.A.10$]{MR0415622}) or it realizes the special case above. We conclude that $G$ does not belong to $\mathcal{MN}$ by the first paragraph of the proof and Proposition~\ref{ends}.
\end{proof}

Observe, that Proposition \ref{ends} gives another obstruction for a finitely generated group in $\mathcal{MN}$ to contain a free group of finite index.

\section{Examples and non-examples}

It is well-known that a finite group belongs to the class $\mathcal{MN}$ if and only if it is nilpotent \cite[5.2.4]{MR648604}. In fact, the finiteness of the group is only required to prove the necessary implication, whilst the sufficient implication remains true for infinite nilpotent groups. 
Indeed, all nilpotent groups belong to $\mathcal{MN}$ since they satisfy the normalizer condition (\emph{i.e.} any proper subgroup is contained properly in its normalizer). 
More generally, any locally nilpotent group belongs to $\mathcal{MN}$, see \cite[12.1.5]{MR648604}. We recall that a group $G$ is \emph{locally nilpotent} if every finitely generated subgroup of $G$ is nilpotent.

\begin{cor}
The first Grigrochuk group $\Gamma$ and its periodic generalisations, namely $\{\Gamma_{\omega}\}$ with $\omega \in \{0,1,2,\}^{\mathbb{N}}$ containing all three letters $\{0,1,2\}$ infinitely many times, belong to $\mathcal{MN}
$ \cite{MR565099,MR764305}. Also Gupta-Sidki $p$-groups \cite{MR759409} and, more generally, all multi-edge spinal torsion groups acting on a regular $p$-ary rooted tree, with $p$ an odd prime, belong to $\mathcal{MN}$.
\end{cor}
\begin{proof}
It was shown by Pervova in \cite{MR1841763} that for the Grigorchuk groups $\Gamma_{\omega}$ with $\omega\in~\{0,1,2\}^{\mathbb{N}}$, such that each of $0,1,2$ appears infinitely many times in $\omega$, maximal subgroups of $\Gamma_{\omega}$ are of finite index. Moreover, it is known that such $\Gamma_{\omega}$ are $2$-groups \cite{MR764305} therefore their finite quotients are nilpotent.

Pervova extended her result in \cite{MR2197824} by showing that also in Gupta-Sidki $p$-groups all maximal subgroups are of finite index. Recently the result for Gupta-Sidki $p$-groups was generalized in \cite{Maxsub_spinal} to torsion multi-edge spinal groups acting on the regular $p$-ary rooted tree, for $p$ an odd prime.
In addition, the Gupta-Sidki $p$-groups and, more generally, torsion multi-edge spinal groups  acting on the regular $p$-ary rooted tree, for $p$ an odd prime, are known to be $p$-groups \cite{Maxsub_spinal}; therefore their finite quotients are nilpotent.
\end{proof}

\begin{remark} \emph{Corollary \ref{direct} provides more examples of groups in the class $\mathcal{MN}
$. For instance, the direct product $\Gamma\times \Gamma$ or $\Gamma\times G_p$, where $\Gamma$ is the first Grigorchuk group and $G_p$ is the Gupta-Sidki $p$-group, is also in $\mathcal{MN}
$. }
\end{remark}

It was shown by Grigorchuk and Wilson in \cite{MR2009443} that every infinite finitely generated subgroup $H$ of $\Gamma$, the first Grigorchuk group, is (abstractly) commensurable with $\Gamma$. In the same paper, the authors showed that for a  group $G$ (abstractly) commensurable with $\Gamma$, all maximal subgroups of $G$ are of finite index. It follows that all finitely generated subgroups of $\Gamma$ are in the class $\mathcal{MN}$. 

Garrido in \cite{Garrido2014} studied the subgroup structure of Gupta-Sidki $3$-group $G_3$ and, following results of Grigorchuk and Wilson, showed that all infinite finitely generated subgroups of $G_3$ have their maximal subgroups of finite index. 
It follows that all finitely generated subgroups of $G_3$ are in $\mathcal{MN}$.

\begin{remark}
\emph{
Another implication of Theorem \ref{properties} is the following: if a finitely generated group $G$ is in $\mathcal{MN}$ and if it is just-infinite (\emph{i.e.} all proper quotients of $G$ are finite) \emph{or} it is just-non-solvable group (\emph{i.e.} all proper quotients of $G$ are solvable) then $G$ is just-non-nilpotent.}

\end{remark}

As observed in Section \ref{freeproduct}, the free group as well as free products of finitely generated groups do not belong to $\mathcal{MN}$. Moreover, a finitely generated group in $\mathcal{MN}$ does not admit a free subgroup of finite index and, more generally, does not admit infinitely many ends. 

Another non-example of the class $\mathcal{MN}$ is a just-non-solvable weakly branch Basilica group $G$ defined by Grigorchuk and Zuk in \cite{MR1902367}.
\begin{cor}
The Basilica group $G$, of automorphisms of a rooted binary tree generated by automorphisms $a=(1,b)$ and $b=(1, a)\sigma$, is not in the class $\mathcal{MN}
$. 
\end{cor}
\begin{proof}
The presentation of the group $G$ obtained in \cite{MR1902367} implies that the quotient $G/N$ with $N=\ll b^2, abab \gg$ is isomorphic to the infinite dihedral group $D_{\infty}$. The latter, as discussed in Section \ref{freeproduct}, does not belong to the class $\mathcal{MN}
$ and it follows that $G$ does not belong to $\mathcal{MN}
$ by Theorem \ref{properties}.
\end{proof}

\subsection{Further examples: a group in $\mathcal{MN}$ containing the free group $F_2$ as a subgroup}
\label{Olshanski}

Existence of a countable simple group without maximal subgroups was discussed by Ol'shanski in \cite[Theorem $35.3$]{MR1191619}. The construction can be adapted to obtain a countable simple group without maximal subgroups which contains the free group $F_2$ as a subgroup (let $G^1$ in the proof to be the free group $F_2$).
Such a group belongs to the class $\mathcal{MN}$ for the simple reason of not having maximal subgroups. It follows that the class $\mathcal{MN}$ is not closed under taking subgroups. 

\appendix
\section{Relation to the Andrews-Curtis conjecture}
\label{appendix}
We will briefly explain here the relation to the famous Andrews-Curtis conjecture mentioned in the introduction. 
Recall that for a group $G$ generated by a set $S$, the following transformations on the set $G^n$, $n\geq 1$, are called the \emph{Andrews-Curtis moves}:
\begin{align*}
R_{ij}^{\pm}(g_1,\dots,g_i,\dots, g_j,\dots, g_n)&=(g_1,\dots,g_{i}g_j^{\pm 1},\dots,g_j,\dots, g_n),\\
L^{\pm}_{ij}(g_1,\dots,g_i,\dots, g_j,\dots, g_n)&=(g_1,\dots,g_j^{\pm 1}g_{i},\dots,g_j,\dots, g_n),\\
 I_j(g_1,\dots,g_j,\dots,g_n)&=(g_1,\dots,g_j^{-1},\dots,g_n),\\
 AC_{j,s}(g_1,...,g_j,...,g_n)&=(g_1,...,s^{-1}g_js,...,g_n)\end{align*} where $1\leq i,j \leq n$, $i\neq j$ and $s\in S$.

The \emph{Andrews-Curtis conjecture} asserts that, for a free group $F_n$ of rank $n\geq 2$ and a free basis $(x_1, \dots, x_k)$ of $F_n$, every normally generating $n$-tuple $(y_1, \dots, y_n)$ of $F_n$ can be transformed into $(x_1, \dots, x_n)$ by a sequence of Andrews-Curtis moves. Note, that very often this conjecture is formulated in terms of balanced presentations of the trivial group. A \emph{balanced presentation} is a presentation of a group that uses a finite number of generators and an equal number of relations. The Andrews-Curtis conjecture stated above is equivalent to the following: for $n\geq 2$ every balanced presentation $\langle x_1, \dots, x_n\mid w_1, \dots, w_n\rangle$ of the trivial group can be transformed using  Andrews-Curtis moves on $\{w_1, \dots, w_n\}$ to the trivial presentation  $\langle x_1, \dots, x_n\mid x_1, \dots, x_n\rangle$.

There have been many attempts to disprove this conjecture. A way to do it would be to find a \emph{potential counter-example}, \emph{i.e.} a normally generating $n$-tuple of $F_n$,
and to show that its image in some quotient $G$ of $F_n$ cannot be transformed into 
that of the basis $(x_1, \dots, x_n)$ by the Andrews-Curtis moves. This motivates the study of the Andrews-Curtis conjecture in finitely generated groups. 

It follows from \cite{MR1684568} that in any abelian group $\mathcal{A}$ the image of a normally-generating tuple of $F_n$ \emph{can} be transformed to that of the basis $(x_1, \dots, x_n)$. 
Furthermore, it was shown in \cite{Myropolska}  that a finitely generated group $G$ in the class $\mathcal{MN}$ satisfies the so-called generalised Andrews-Curtis conjecture\footnote{\emph{i.e.}
two normally generating $n$-tuples of $G$ are connected by the Andrews-Curtis moves if and only if their images in the abelianization $G/[G,G]$ are connected by the Andrews-Curtis moves.} and thus, by the remark above, cannot confirm potential-counter examples to the conjecture. We refer the reader to \cite[Chapter 3]{thesis} for further details.

\bibliography{Bibliography}
\bibliographystyle{alpha}

\end{document}